\documentclass[12pt,letterpaper,reqno]{amsart}

\usepackage{times}
\usepackage[T1]{fontenc}
\usepackage{mathrsfs}
\usepackage{latexsym}
\usepackage[dvips]{graphics}
\usepackage{epsfig}
\usepackage{amsmath,amsfonts,amsthm,amssymb,amscd}
\input amssym.def
\input amssym.tex

\newcommand\be{\begin{equation}}
\newcommand\ee{\end{equation}}
\newcommand\bea{\begin{eqnarray}}
\newcommand\eea{\end{eqnarray}}
\newcommand\bi{\begin{itemize}}
\newcommand\ei{\end{itemize}}
\newcommand\ben{\begin{enumerate}}
\newcommand\een{\end{enumerate}}
\newcommand\bc{\begin{center}}
\newcommand\ec{\end{center}}
\newcommand\ba{\begin{array}}
\newcommand\ea{\end{array}}

\newtheorem{thm}{Theorem}[section]

\newtheorem{lem}[thm]{Lemma}
\newtheorem{prop}[thm]{Proposition}

\theoremstyle{definition}

\begin{document}

\title{Answers to two questions posed by Farhi concerning additive bases}

\author{Peter Hegarty}
\email{hegarty@chalmers.se} \address{Mathematical Sciences,
Chalmers University Of Technology and University of Gothenburg,
41296 Gothenburg, Sweden}


\date{\today}

\begin{abstract} Let $A$ be an asymptotic basis for $\mathbb{N}$ and $X$ a 
finite subset of $A$ such that $A \backslash X$ is still an asymptotic
basis. Farhi recently proved a new batch of upper bounds for the order of
$A \backslash X$ in terms of the order of $A$ and a variety of parameters
related to the set $X$. He posed two questions concerning possible
improvements to his bounds. In this note, we answer both questions.  
\end{abstract}


\maketitle

\setcounter{equation}{0}

\setcounter{equation}{0}

\section{Introduction}

Let $\mathbb{S}$ be a countably infinite abelian semigroup and $\mathbb{T}$ an
infinite subset of 
$\mathbb{S}$. A subset $A \subseteq \mathbb{S}$ is said to be  
an {\em asymptotic basis} for $\mathbb{T}$ if, for
some integer $h$, the $h$-fold sumset $hA$ contains all but finitely many
elements of $\mathbb{T}$. The least such $h$ is called the {\em order} of 
the asymptotic basis 
$A$ and is commonly denoted $G(A)$. The most natural setting is when 
$\mathbb{S} = \mathbb{Z}$, $\mathbb{T} = \mathbb{N}$ and $|A \cap
\mathbb{Z}_{-}| < \infty$. This will be the setting for the remainder of our
discussion, so henceforth we use the term $\lq$asymptotic basis' without
explicit reference to the sets $\mathbb{S}$ and $\mathbb{T}$.  
\par A classical result of Erd\H{o}s and Graham \cite{2} states that if $A$
is an asymptotic basis and $a \in A$, then $A \backslash \{a\}$ is still an
asymptotic basis if and only if $\delta(A) = 1$ where, for any set $S$ of 
integers, one denotes 
\be
\delta(S) := {\hbox{GCD}} \{x-y : x,y \in S \}.
\ee 
Moreover, in that case, the order $G(A \backslash \{a\})$ can be bounded by a 
function of $G(A)$ only. For a positive integer $h$, one denotes by $X(h)$ the
maximum possible order of an asymptotic basis $A \backslash \{a\}$, where
$G(A) \leq h$ and $G(A \backslash \{a\}) < \infty$. This function has been
the subject of a considerable amount of attention. The best-known universal
lower and upper bounds are both due to Plagne \cite{12}, who showed that
\be
\lfloor \frac{h(h+4)}{3} \rfloor \leq X(h) \leq \frac{h(h+1)}{2} + \lceil
\frac{h-1}{3} \rceil.
\ee
A more precise determination of the asymptotic behaviour of the
quotient $X(h)/h^2$ is the major open problem in this area. 
\par This note is concerned with a generalisation of the basic problem 
first introduced by Nash and Nathanson (see \cite{8}, \cite{10}, \cite{11}). If
$A$ is an asymptotic basis and $X$ a finite subset of $A$, then the 
Erd\H{o}s-Graham result is easily generalised to the statement that 
$A \backslash X$ is still a basis if and only if $\delta (A \backslash X) = 1$,
and in that case that $G(A \backslash X)$ is bounded by a function of 
$G(A)$ and $|X|$ only. For positive integers $k$ and $h$, one denotes by 
$X_{k}(h)$ the maximum possible order of an asymptotic basis $A \backslash X$,
where $G(A) \leq h$, 
$X \subseteq A$, $|X| = k$ and $G(A\backslash X) < \infty$. 
One is primarily interested in the behaviour of this 
function as $h \rightarrow \infty$ for a fixed $k$ (the reverse situation has 
also been studied, but is not our concern here). In that case, it is 
known that 
\be
\frac{4}{3} \left( \frac{h}{k+1} \right)^{k+1} \lesssim X_{k}(h) 
\lesssim \frac{h^{k+1}}{(k+1)!}.
\ee
Here, the lower and upper bounds were established in \cite{6} and \cite{9}
respectively. The basic point is that, for fixed $k$, the function $X_{k}(h)$
exhibits polynomial growth in $h$ of degree $k+1$.   
\par In a recent paper \cite{3}, Farhi sought universal upper bounds for
orders $G(A \backslash X)$, which were polynomial of fixed degree in both
$G(A)$ and in some $\lq$natural' parameter associated to the set $X$, other
than simply its size. He obtained three such bounds and, in order to
state his results, we need some notation. Let $X$ be a finite set of
integers. The diameter
of $X$, denoted diam$(X)$, is the difference between the largest and smallest
elements of $X$. We define 
\be
d = d(X) := \frac{{\hbox{diam}}(X)}{\delta(X)}.
\ee
Now suppose $A$ is an infinite set of integers containing $X$. One sets
\be
\eta = \eta(A,X) := \min_{\stackrel{a,b \in A \backslash X, \; a \neq b}{|a-b| \geq {\hbox{diam}}(X)}} |a-b|
\ee
and
\be
\mu = \mu(A,X) := \min_{y \in A \backslash X} {\hbox{diam}}(X \cup \{y\}).
\ee
Then the following results are proven in \cite{3} :

\begin{thm} {\bf (Farhi)}
Let $A$ be an asymptotic basis with $G(A) \leq h$ and $X$ a finite 
subset of $A$ such that $G(A \backslash X) < \infty$. Let the quantities
$d = d(X)$, $\eta = \eta(A,X)$ and $\mu = \mu(A,X)$ be as defined above. Then 
\be
G(A \backslash X) \leq \frac{h(h+3)}{2} + d \left[ \frac{h(h-1)(h+4)}{6} 
\right], 
\ee
\be
G(A \backslash X) \leq \eta (h^2 - 1) + h + 1
\ee
and
\be
G(A \backslash X) \leq \frac{h\mu (h\mu + 3)}{2}.
\ee 
\end{thm}

At the end of his paper, Farhi posed the following two questions :
\\
\\
{\bf Question 1.} Can one improve the upper bound in (1.7) to a function which
is linear in $d$ and quadratic in $h$ ?
\\
\\
{\bf Question 2.} Can one improve the upper bound in (1.9) to a function 
which is linear in $\mu$ and quadratic in $h$ ?
\\
\\
Our two main results answer these questions, the first in the negative and
the second in the affirmative. More precisely, we shall prove the
following two theorems :

\begin{thm} 
With notation as in Theorem 1.1, let $f(h)$ be any function
such that $G(A \backslash X) \leq d \cdot f(h)$, for any possible choice of 
$A$ and $X$. Then, for any fixed integer $d$, as $h \rightarrow \infty$ we must
have $f(h)/h^3 \gtrsim 1/27$.  
\end{thm}

\begin{thm}
With notation as in Theorem 1.1, we have 
\be
G(A \backslash X) \leq 4h(2h\mu + 1).
\ee
\end{thm}

The proofs of these two results will be presented in Sections 2 and 3 
respectively. We close this section by putting our results in context. 
Though the following discussion will be familiar to experts in the area, we 
think it also serves to highlight central features of the problems at hand 
in a way which is not always apparent in the existing literature. 
\\
\\
We recall some standard notation and terminology. If $A,B$ are two sets of
integers, then we write $A \sim B$ to denote that the 
symmetric difference $A \Delta B$ is finite. If $A \subseteq \mathbb{Z}$ and
$n \in \mathbb{N}$ then $A^{(n)}$ denotes the set of all non-negative
integers $x$ such that $x \equiv a \; ({\hbox{mod $n$}})$ for some $a \in A$. 
Finally, the {\em lower asymptotic density} of a set
$A \subseteq \mathbb{Z}$, denoted {\bf $\underline{d}$}$(A)$, is defined as
\be
{\hbox{{\bf $\underline{d}$}}}(A) := \liminf_{n \rightarrow +\infty}
\frac{|A \cap \{1,...,n\}|}{n}.
\ee
The proofs of good upper bounds for the functions $X_{k}(h)$
employ the classical results of Kneser concerning the structure of 
sets of integers with $\lq$small doubling'. The basic crucial result is
the following :

\begin{thm} {\bf (Kneser)}
Let $A$ be a set of integers with $|A \cap \mathbb{Z}_{-}| < \infty$. 
Suppose that {\bf $\underline{d}$}$(A) > 0$ and that 
{\bf $\underline{d}$}$(2A) < 2 {\hbox{{\bf $\underline{d}$}}}(A)$. Then   
there exists a positive integer $n$ such that $2A \sim (2A)^{(n)}$. 
\end{thm}

Lower bounds like those in (1.2) and (1.3) are obtained by construction of 
explicit examples, each based on a set of integers with
small doubling. For simplicity, let's first concentrate on the case of (1.2), 
which has received the greatest attention. 
There are basically two types of construction in existence. 
In each case, the set $A$ is the union of a set $A^{*}$ with 
small doubling and a single element $x$, whose removal increases the 
order of the basis from $h$ to around $h^{2}/3$. Note that, without loss of 
generality
$x = 0$, since the order of an asymptotic basis is translation invariant.  
\par In the one type of construction, 
the set $A^{*}$ is a union of two arithmetic progressions 
with a common modulus, in the other it is a so-called {\em Bohr set}. 
Significantly, it is known that the lower bound in (1.2) cannot be
raised by a construction of either type : see Lemma 15 and Conjecture 21 of
\cite{12} for details{\footnote{Conjecture 21 in Plagne's paper is actually a 
theorem, having already been proven a number of years earlier by 
Hsu and Jia \cite{5}.}. Personally, I believe that these types of 
constructions are optimal, in other words that the lower bound in 
(1.2) is the exact value of $X(h)$. Closing the gap in our current knowledge
seems to be intimately connected to a better understanding of the 
structure of sets $A$ satisfying {\bf $\underline{d}$}$(2A) <
\sigma {\hbox{{\bf $\underline{d}$}}}(A)$, where the doubling constant $\sigma$
is slightly bigger than two. Kneser-type structure theorems are known in this
setting - they are basically due to Freiman, but see \cite{1} for more
state-of-the-art formulations - and while they support the intuition that
Bohr sets and unions of arithmetic progressions should yield optimal
constructions, the structure theorems which have actually been proven to date 
seem to be too weak to definitively yield such a conclusion. 
\par In the more general case of (1.3), there is also greater uncertainty 
regarding the lower bound, and this seems to be intimately connected to 
the gaps in our current understanding of the so-called {\em postage stamp
problem} in finite cyclic groups : see \cite{4} for a discussion of this
problem. 
\\
\\
In light of the above observations, it should be no surprise that Kneser's 
theorem is also the crucial element in Farhi's proof of Theorem 1.1 and that
our proofs of Theorems 1.2 and 1.3 are based respectively on an 
explicit construction reminiscent of those discussed above, and on a 
more judicious application of Kneser's result. As will be discussed
briefly in Section 4, our results are also optimal up to constant factors and
we suppose that in these instances also, 
the precise determination of the right constants 
will demand a better understanding of fundamental problems in additive number 
theory like the structure of sets with small 
doubling and the postage stamp problem.

\setcounter{equation}{0}
\section{Proof of Theorem 1.2}

In our construction, the 
set $X$ will be an arithmetic progression. Note that, in that case, 
$d(X) = |X| - 1$. So let $d,k$ be positive integers with $k \geq 2$ 
and put $h = 3k$. Set 
\be
X = \{0,k,2k,...,dk\}.
\ee
Take $n = dk^3$ and set 
\be
A^{*} := \{x \in \mathbb{N} : x \; ({\hbox{mod $n$}}) \in \{1,dk^2\} \}.
\ee
Finally, take $A = A^{*} \cup X$. We claim that both $A$ and $A^{*}$ are
asymptotic bases and that 
\be
G(A) = h-2, \;\;\;\; G(A^{*}) = \frac{dh^3}{27} - 1.
\ee
Note that Theorem 1.2 follows directly from these equalities, 
so it just remains
to verify them.
\par Concerning $A^{*}$, this is an asymptotic basis if and only if 
$\{1,dk^2\}$ is a basis for $\mathbb{Z}/n\mathbb{Z}$. The latter is indeed
the case, since GCD$(n,dk^2-1)=1$. It is then 
clear that $G(A^{*}) = n-1 = dk^3 - 1 = \frac{dh^{3}}{27} - 1$, as desired.
\par Turning to $A$, we first show that each of the numbers 
$0,1,...,n-1$ can be represented as a sum of at most $h-2 = 3k-2$ elements 
from 
the set $Y := \{0,1,k,2k,...,dk,dk^2\}$. First of all, the set 
$kX$ contains all
multiples of $k$ from $0$ up to and including $dk^2$. Hence, if $0 \leq m <
dk^2$ then we can write $m = x + t \cdot 1$, 
where $x \in kX$ and $0 \leq t < k$.
Secondly, if $dk^2 \leq m < dk^3$ we can write $m = s \cdot (dk^2) + x + 
t \cdot 1$, where $0 \leq s,t < k$ and $x \in kX$. 
It follows that $\{0,1,...,n-1\} 
\subseteq (3k-2)Y$, as claimed. From this fact, it is easily deduced that
$G(A) \leq 3k-2$. One just has to careful with integers that are
congruent to an element of $kX$ modulo $n$. But since $n = k \cdot (dk^2)$, all
sufficiently large such numbers lie in $(2k)A$. Since $k \geq 2$, we have
$2k \leq 3k-2$ and thus $G(A) \leq 3k-2$, as desired. 
In fact, we have equality since a number congruent 
to $-1 \; ({\hbox{mod $n$}})$ 
is easily seen not to be representable as a sum of
strictly fewer than $3k-2$ elements of $A$.  
\par This completes the proof of Theorem 1.2.   

\setcounter{equation}{0}
\section{Proof of Theorem 1.3}

The proof follows the argument employed by Farhi to prove (1.9), but makes use
of a couple of extra observations which we first present. To begin with :

\begin{lem}
Let $\alpha$ be positive real number and $S$ a set of non-negative
integers with the property that, for every $n \gg 0$, there exists some
$s \in S$ such that $|s-n| \leq \alpha$. Then $\underline{d}(S) \geq 
\frac{1}{2\lceil \alpha \rceil + 1}$. 
\end{lem}

\begin{proof} 
Put $t := 2 \lceil \alpha \rceil + 1$ and 
let $n$ be a very large positive integer. Divide the integers 
$1,2,...,t \cdot \lfloor n/t \rfloor$ into 
$\lfloor n/t \rfloor$ subsets of $t$ 
consecutive integers each. 
The assumption of the lemma implies that all but $O(1)$
of these subsets contain at least one element from $S$. Hence
$|S \cap \{1,...,n\}| \geq \lfloor \frac{n}{t} \rfloor - O(1)$, 
and it follows immediately
that $\underline{d}(S) \geq 1/t$. 
\end{proof}

Our second observation will be an explicit upper bound for the order
of an asymptotic basis of a given lower density. We shall make use of
Theorem 1.4 plus a result concerning bases in finite cyclic
groups $\mathbb{Z}/n\mathbb{Z}$. Here a subset $A \subseteq 
\mathbb{Z}/n\mathbb{Z}$ is called a basis if $hA = \mathbb{Z}/n\mathbb{Z}$ for
some $h \in \mathbb{N}$ and the least such $h$ is called the order of $A$. 
To further distinguish the notion of basis from that of asymptotic basis 
(which makes no 
sense in the finite setting), we denote the order in the former case
by $\rho(A)$. The following result is part of Theorem 2.5 of \cite{7} :

\begin{thm} {\bf (Klopsch-Lev)}
Let $n \in \mathbb{N}$ and $\rho \in [2,n-1]$. Let $A$ be a basis for 
$\mathbb{Z}_n$ such that $\rho(A) \geq \rho$. Then 
\be
|A| \leq \max \left\{ \frac{n}{d} \left( \lfloor \frac{d-2}{\rho - 1} \rfloor
+ 1 \right) : d | n, \; d \geq \rho + 1 \right\}.
\ee
\end{thm}

From (3.1) it is easily checked to follow that, if $A$ is a basis for 
$\mathbb{Z}/n\mathbb{Z}$, then 
\be
|A| \cdot \rho(A) < 2 n.
\ee
Now we can state the result we shall use :

\begin{lem}
Let $S \subseteq \mathbb{Z}$ satisfy $|S \cap \mathbb{Z}_{-}| < \infty$. 
Suppose that $\underline{d}(S) > 0$ and that $S$ is an asymptotic 
basis. Then $G(S) \leq 4/\underline{d}(S)$.
\end{lem}

\begin{proof}
Let $k$ be the unique non-negative integer
satisfying $2^{k} \alpha \leq 1 < 2^{k+1} \alpha$. There is a 
smallest integer $j \in \{0,...,k\}$ such that $\underline{d}(2^{j}S)
\geq 2^{j} \alpha$ and $\underline{d}(2^{j+1}S) < 2 \underline{d}(2^{j}S)$.
Set $T := 2^{j} S$ and $\beta := 2^{j} \alpha$, so that $\underline{d}(T)
\geq \beta$ and $\underline{d}(2T) < 2 \underline{d}(T)$. By Theorem 1.4, 
there thus exists a positive integer $n$ such that $2T \sim (2T)^{(n)}$. 
Let $\mathscr{T} \subseteq 
\mathbb{Z}/n\mathbb{Z}$ be the image of $2T$ under the natural projection. 
Then the order of $2T$ as an asymptotic basis
is at most the order of $\mathscr{T}$ 
as a basis in $\mathbb{Z}/n\mathbb{Z}$. Eq. (3.2) implies that
$\rho(\mathscr{T})
< 2n/|\mathscr{T}| \leq 2/\beta$ and hence $G(2T) \leq 2/\beta$ also. 
Finally, then, $G(S) \leq 2^{j+1} G(2T) \leq \left( 
\frac{2\beta}{\alpha} \right) \left( \frac{2}{\beta} \right) = 
\frac{4}{\alpha}$, as required. 
\end{proof}

We can now prove Theorem 1.3. Let $A$ be an asymptotic basis of order $h$, $X$
a finite subset of $A$ such that $G(A \backslash X) < \infty$ and let the 
parameter $\mu$ be as defined in (1.6). By translation invariance, there
is no loss of generality in assuming that $0 \in A \backslash X$ and that 
$\mu = {\hbox{diam}}(X \cup \{0\})$. Put $A^{*} := A \backslash X$. 
If one now follows the proof of (1.9) in
\cite{3}, one readily verifies that what is actually established there is 
that, for every $n \gg 0$, there is some element $a \in hA^{*}$ such that
$|n - a| \leq h\mu$. By Lemma 3.1, it follows that $\underline{d}(hA^{*})
\geq \frac{1}{2h\mu + 1}$. Since $hA^{*}$ is an asymptotic basis, Lemma 
3.3 then implies that its order is at most $4(2h\mu + 1)$. Hence, 
$G(A^{*}) = G(A \backslash X) \leq 4h(2h\mu + 1)$, as required. 
     
\setcounter{equation}{0}
\section{Concluding remarks}

For each positive rational number $d$ we can define the function $X_{d}(h)$
by 
\be
X_{d}(h) = \max_{(A,X)} \frac{1}{d} \; G(A \backslash X),
\ee
the maximum being taken over all pairs $(A,X)$ where $A$ is an asymptotic basis
of order at most $h$ and $X$ is a finite subset of $A$ such that 
$G(A \backslash X) < \infty$ and $d(X) = d$. From (1.7) and Theorem 1.2
it follows 
that, for each integer value of $d$, as $h \rightarrow \infty$ then 
\be
\frac{1}{27} \lesssim \frac{X_{d}(h)}{h^3} \lesssim \frac{1}{6}.
\ee
Similarly, for each positive integer $\mu$, the function $X_{\mu}(h)$ can be
defined by 
\be
X_{\mu}(h) = \max_{(A,X)} \frac{1}{\mu} \; G(A \backslash X),
\ee
where this time the maximum is taken with respect to finite sets $X$
satisfying $\mu(X) = \mu$. From (1.10) one concludes that, for each 
fixed $\mu$ and as $h \rightarrow \infty$, 
\be
\frac{X_{\mu}(h)}{h^2} \lesssim 8.
\ee
For a lower bound, we have

\begin{prop}
For every $\mu \in \mathbb{N}$, as $h \rightarrow \infty$ we have 
\be
\frac{X_{\mu}(h)}{h^2} \gtrsim \frac{1}{4}.
\ee
\end{prop}

\begin{proof}
If $\mu(X) = 1$ then the set $X$ must consist of a single element. Then from
(1.2) it follows that $X_{\mu}(h)/h^2 \gtrsim 1/3$ in this case.
\par Now let integers $\mu, h \geq 2$ be given. Take $X = \{0,1\}$, 
$n = h(h-1)\mu + 1$, $A^{*} = \{x \in \mathbb{N} : x \; ({\hbox{mod $n$}}) \in
\{\mu,h\mu\} \}$ and $A = A^{*} \cup X$. Clearly, $\mu(X) = \mu$. 
Secondly, $A^{*}$ is an asymptotic basis, since GCD$(n,(h-1)\mu) = 1$ and
$G(A^{*}) = n-1 = h(h-1)\mu$. Thirdly, it is easy to check that 
$\{0,1,...,n-1\} \subseteq (2h+\mu - 4)Y$, where $Y = \{0,1,\mu,h\mu\}$. 
This in turn is easily seen to imply that $G(A) \leq 2h+\mu-4$ (in fact,
$G(A) = 2h-2$ when $\mu = 2$ and $G(A) = 2h + \mu - 5$ when $\mu \geq 3$). 
Letting $h \rightarrow
\infty$ we deduce (4.5). 
\end{proof}

It remains to obtain tighter bounds than those given in (4.2), (4.4) and (4.5).
The lower bounds in (4.2) and (4.5) can probably be improved by more
judicious constructions similar to those given in this paper. 
However, a satisfactory solution of the whole problem 
will, we speculate, require 
significant advances in our understanding of, on the one hand, the
structure of sets with small doubling and, on the other, of the
postage stamp problem. 

\section*{Acknowledgements}
This work was completed while the author was visiting the Mittag
Leffler Institute in Djursholm, Sweden, and I thank them for their
hospitality. My research is supported by a grant from the Swedish 
Science Research Council (Vetenskapsr\aa det).

\ \\

\end{document}